\documentclass[letterpaper, 10 pt, conference]{ieeeconf}
\usepackage{cite}
\usepackage{amsmath,amssymb,amsfonts}
\usepackage{graphicx}
\usepackage{textcomp}
\usepackage{hyperref}
\hypersetup{hidelinks=true}
\usepackage{lipsum}
\usepackage{pifont}
\usepackage{stfloats}
\usepackage{algorithm}
\usepackage{algpseudocode}

\newtheorem{secthm}{Theorem}[section]
\newtheorem{seccor}[secthm]{Corollary}
\newtheorem{seclem}[secthm]{Lemma}

\newtheorem{secprob}[secthm]{Problem}
\newtheorem{secdefn}[secthm]{Definition}

\newcommand{\bE} { {\mathbb E}}
\newcommand{\bP} { {\mathbb P}}
\newcommand{\bR} { {\mathbb R}}

\newcommand{\tr} { {\operatorname{Tr} }}

\newcommand{\cF} { {\mathcal F}}

\newcommand{\cD} { {\mathcal D}}

\def\red{\hfill $\lhd$}

\def\BibTeX{{\rm B\kern-.05em{\sc i\kern-.025em b}\kern-.08em
    T\kern-.1667em\lower.7ex\hbox{E}\kern-.125emX}}
\begin{document}
\title{Distributionally Robust Model Order Reduction for Linear Systems}
\author{Le Liu, Yu Kawano, Yangming Dou and Ming Cao
 \thanks{The work of Liu and Cao was supported in part by the European Research Council (ERC-CoG-771687). The work of Yu Kawano was supported in part by JSPS KAKENHI Grant Number JP22KK0155.}
\thanks{Le Liu, Yangming Dou and Ming Cao are with the Faculty of Science and Engineering, Univerisity of Groningen, 9747 AG Groningen, The Netherlands {\tt \small \{le.liu, y.dou, m.cao\}@rug.nl}}
    \thanks{Yu Kawano is with the Graduate School of Advance Science and Engineering, Hiroshima Univeristy, Higashi-hiroshima 739-8527, Japan 
        {\tt\small ykawano@hiroshima-u.ac.ip}}
}

\maketitle

\begin{abstract}
In this paper, we investigate distributionally robust model order reduction for linear, discrete-time, time-invariant systems. The external input is assumed to follow an uncertain distribution within a Wasserstein ambiguity set. We begin by considering the case where the distribution is certain and formulate an optimization problem to obtain the reduced model. When the distribution is uncertain, the interaction between the reduced-order model and the distribution is modeled by a Stackelberg game. To ensure solvability, we first introduce the Gelbrich distance and demonstrate that the Stackelberg game within a Wasserstein ambiguity set is equivalent to that within a Gelbrich ambiguity set. Then, we propose a nested optimization problem to solve the Stackelberg game. Furthermore, the nested optimization problem is relaxed into a nested convex optimization problem, ensuring computational feasibility. Finally, a simulation is presented to illustrate the effectiveness of the proposed method.
\end{abstract}


\section{Introduction}
Large, complex engineering systems exhibit in various forms, in e.g. robotic, energy and manufacturing systems\cite{mesbahi2010graph, antoulas2005approximation}. Their high dimensionality presents significant challenges in analysis, design, and real-time implementation. Model order reduction provides a viable solution by offering a lower-dimensional approximation. In complex systems, the inputs are often subject to randomness arising from various uncertainties, such as environmental factors, measurement errors and unpredictable disturbances \cite{weinmann2012uncertain}. This inherent randomness can significantly impact system behavior and performance. By incorporating random inputs into the system analysis, we can formulate a more meaningful model order reduction problem, resulting in a better approximation of the original system.

\smallskip

\subsubsection*{Literature review}
Model order reduction for linear systems has been addressed using various methods, including balanced truncation \cite{sandberg2004balanced, cheng2019balanced}, projection-based procedures \cite{benner2015survey}, moment matching \cite{astolfi2010model}, Hankel-norm-minimization \cite{safonov1990optimal} and convex optimization \cite{ibrir2017h, ibrir2018projection, yu2022h2}. Most of these approaches focus on $H_2$ or $H_{\infty}$ performance. When the input follows a distribution, a moment matching method to capture the statistical properties is explored in \cite{scarciotti2021moment}. The balanced truncation method is also employed to approximate statistical properties in \cite{benner2011lyapunov}. However, these studies primarily focus on model order reduction for approximating statistical properties. When the inputs are random but measurable, the model order reduction problem is closely related to $H_2$ performance \cite{hyland1985optimal, bernstein1986optimal} once the distribution of the inputs is known. When the distribution is uncertain, the research on model order reduction remains limited.

\smallskip

\subsubsection*{Contribution}
In this paper, we address the problem of model order reduction for linear, discrete-time, time invariant systems, where the external input is assumed to follow a distribution within a Wasserstein ambiguity set. We first develop an optimization-based approach to perform model order reduction when the input distribution is known. Subsequently, we formulate the distributionally robust model order reduction problem as a Stackelberg game, capturing the interplay between the reduced-order model and the uncertain distribution. Finally, to ensure computational feasibility, we relax the game formulation into a nested convex optimization problem, enabling efficient and tractable solutions.


\smallskip

\subsubsection*{Organization}
Section \ref{sec:prob} states problem definition. Section \ref{sec:result} presents our proposed solution, which is relaxed into a tractable solution in Section \ref{sec:algo}. A simulation result is provided in \ref{sec:simu}.  Conclusions are given in
 Section \ref{sec:con}.


\smallskip

\subsubsection*{Notation}
The sets of real numbers is denoted by $\bR$. For $P \in \bR^{n \times n}$, $P \succ 0$ (resp. $P \succeq 0$) means that $P$ is symmetric and positive (resp. semi) definite. We also use $\mathbb{S}_{++}^n$ (resp. $\mathbb{S}_{+}^n$) to denote the space of positive (resp. semi) definite matrices. The space of symmetric space is denoted by $\mathbb{S}^n$. A probability space is denoted by $(\Omega, \cF, \bP )$, where $\Omega$, $\cF$, and $\bP $ denote the sample space, $\sigma$-algebra, and probability measure, respectively. $\mathcal{N}(\mu, Q)$ represents a Gaussian distribution with mean $\mu$ and covariance $Q$. The expectation of a random variable is denoted by $\bE[\cdot]$.  Given two probability measures $\mu, \nu \in \mathcal{P}_2(\mathbb{R}^n)$, the 2-Wasserstein distance between $\mu$ and $\nu$ is defined by
\[
\mathcal{W}_2(\mu, \nu) = \inf_{\pi \in \Pi(\mu, \nu)} \left( \int_{\mathbb{R}^n \times \mathbb{R}^n} \|x - y\|^2 \, d\pi(x, y) \right)^{\frac{1}{2}},
\]
where $\Pi(\mu, \nu)$ is the set of all couplings of $\mu$ and $\nu$, i.e., the set of all probability measures $\pi$ on $\mathbb{R}^n \times \mathbb{R}^n$ such that their marginals satisfy $\pi(A \times \mathbb{R}^n) = \mu(A), \quad \pi(\mathbb{R}^n \times B) = \nu(B), \quad \forall A, B \in \mathcal{B}(\mathbb{R}^n)
$, where $\mathcal{B}(\mathbb{R}^n)$ is the Borel $\sigma$-algebra \cite{bogachev2007measure}.  

For two multivariate Gaussian distributions $\bP_1 = \mathcal{N}(m_1, \Gamma_1) \quad \text{and } \bP_2 = \quad \mathcal{N}(m_2, \Gamma_2),
$
it holds that \cite{panaretos2019statistical}
\begin{align*}
\mathcal{W}_2^2(\bP_1, \bP_2) &= \| m_1 - m_2 \|^2 + \tr ( \Sigma_1 \\
+& \Sigma_2 - 2 (\Sigma_1^{1/2} \Sigma_2 \Sigma_1^{1/2})^{1/2} ).
\end{align*}

\section{Distributionally robust model order reduction}
Traditional model order reduction techniques often assume the precise knowledge of input distributions \cite{bernstein1986optimal}, a condition that may not hold in real-world applications. For instance, in a power system, model order reduction could be used to approximate the real time voltage output once the current input is measured. While the input current is expected to be of precise magnitude and frequency, random fluctuations are inevitable. Even if random fluctuations can be modeled by a distribution, accurately characterizing them remains challenging.  A common approach to approximating the distribution is through Gaussian approximation. However, real-world randomness rarely conform to a perfect Gaussian distribution \cite{rahimian2022frameworks}. Therefore, to effectively analyze system dynamics under random inputs with uncertain distributions, a distributionally robust model order reduction approach is desirable.

In this section, we formulate a distributionally robust model order reduction problem for linear, discrete time, time invariant systems. 
\label{sec:prob}
Consider the following system,
\begin{align}
\label{sys:origin}
    \Sigma : \begin{cases}
         x_{k+1} &= A x_k + Bu_k,  \\
         y_{k} &= C x_{k},
    \end{cases}
\end{align}
where $x_k \in \bR^n$, $u_k \in \bR^m$ and $y_k \in \bR^p$ are states, control inputs and outputs of the system, respectively. Similar to previous model order reduction problem \cite{bernstein1986optimal}, we also assume $A$ is asymptotically stable.

In this paper, it is assumed that the input follows a specific independent and identical distributed (i.i.d.) distribution, denoted by $\bP$. When the distribution is certain, standard projection-based model order reduction techniques can be applied \cite{bernstein1986optimal} and $H_2$ performance of the model order reduction techniques can be guaranteed. However, in cases where the distribution is uncertain, such guarantees no longer hold.

To formulate a tractable problem, we assume that the distribution is uncertain but constrained within a Wasserstein ambiguity set. i.e.,
\begin{align}
\label{eq:P_dis}
    \bP \in \mathbb{W}_{\rho}(\bar{\bP}):= \{\bP \mid \bE_{x \sim \bP}[x] = 0, \mathcal{W}_2(\bP, \bar{\bP}) \leq \rho \},
\end{align}
where $\bar{\bP}$ is a nominal Gaussian distribution with zero mean and covariance $\bar{Q}$, $\rho \geq 0$ is a the radius of the Wasserstein ambiguity set.  Although this paper considers the  input with $0$ mean, it is possible to generalize the result to the input with any finite fixed mean. Furthermore, the formulation on the input can be regarded as a specific stochastic counterpart of the input presented in \cite{astolfi2010model}.

Given that the distribution of the external input is uncertain, it is necessary to develop a model order reduction method that guarantees the performance for any distribution within $\mathbb{W}_{\rho}(\bar{\bP})$. The central problem addressed in this paper is summarized as follows,

\begin{secprob}
    Given a linear, discrete-time, time invariant system $\Sigma$ as in \eqref{sys:origin}, find a reduced order model
    \begin{align}
\label{sys:r1}
    \hat{\Sigma} : \begin{cases}
        \hat{x}_{k+1} &= \hat{A} \hat{x}_k +  \hat{B}{u}_k, \\ 
        \hat{y}_k &= \hat{C} \hat{x}_k,
    \end{cases}
\end{align}
where $\hat{x}_k \in \bR^{r}$ such that the following objectives are achieved:
\begin{itemize}
    \item The reduced-order system \eqref{sys:r1} is asymptotically stable.
    \item The reduced-order state $\hat{x}_{k}$ has dimension $r \le n$.
    \item  The asymptotical approximation error $\lim_{k \to \infty} \bE[(y_k - \hat{y}_k)^{\top}(y_k - \hat{y}_k)]$ is small for all $\bP$ satisfies \eqref{eq:P_dis}.
\end{itemize}
\end{secprob}
\section{Main results}
\label{sec:result}
In this section, we develop a distributionally robust reduced order model in \eqref{sys:r1} step by step. Inspired by popular approaches in distributionally robust optimization \cite{rahimian2022frameworks} , we first transform the model order reduction problem into an optimization problem once $\bP$ is known. Then, we formulate the distributionally robust model order reduction problem as a Stackelberg game for uncertain $\bP \in \mathbb{W}_{\rho}(\bar{\bP})$. Finally, we solve the Stackelberg game by a nested optimization problem. The computational challenges associated with this approach will be addressed in the following section.

\subsection{Model Order Reduction for Certain Distribution}
Following \cite{ibrir2018projection,yu2022h2}, we first propose an optimization problem that facilitates the design of the reduced-order system once the covariance of $\bP$ denoted as $Q \in \mathbb{S}_{+}^n$ is known. A Stackelberg game that captures distributional robustness can then be formulated based on the optimization problem in Subsection \ref{subsec:drmor}.

The following theorem presents a method for model order reduction with a known $Q$ by solving an optimization problem.
\begin{secthm}
\label{thm:opt}
     Consider the system \eqref{sys:origin}, where the inputs follow a known i.i.d. distribution $\bP$ with covariance $Q \in \mathbb{S}_{+}^m$. Given a reduced dimension $1 \leq r < n$, if there exist a scalar $\gamma \geq 0$, a positive definite matrix $P_{\cD}$, a positive semi-definite matrix $Z_{\cD}$ such that the optimization problem \eqref{eq:opt1} is solvable,
    \begin{subequations}
    \label{eq:opt1}
         \begin{align}
        &\min_{P_{\cD}, Z_{\cD}, \gamma}  \  \gamma \nonumber \\
        &\text{s.t.} \label{con_1}\   P_{\cD} = \begin{bmatrix}
            P_{1,\cD} & P_{2,\cD} \\
            P_{2,\cD}^{\top} & P_{3,\cD}
        \end{bmatrix} \succ 0, P_{1,\cD} \in \mathbb{S}^{n}_{+}, P_{3, \cD} \in \mathbb{S}^{r}_{+},\\
       \label{con_2} &\Psi(P_{\cD}, Z_{\cD}, Q) = \begin{bmatrix}
            \Psi_1 & \Psi_2 \\
            \Psi_2^{\top} & \Psi_3
        \end{bmatrix} \prec 0, \\
       \label{con_4} &  Z_{\cD} = P_{2,\cD} P_{3, \cD}^{-1} P_{2,\cD}^{\top}, \operatorname{rank}(Z_{\cD}) = r,\\
    \label{con_5} &\tr\left(C ( P_{1, \cD} -  Z_{\cD}) C^{\top}\right) \leq \gamma,
    \end{align}
    \end{subequations}
      where $\Psi_1 = AP_{1,\cD}A^{\top} - P_{1,\cD} + B Q B^{\top}, \\
        \Psi_2 =  AZ_{\cD}A^{\top} - P_{1,\cD} + B Q B^{\top},  \\
        \Psi_3 = AZ_{\cD}A^{\top} - P_{1,\cD} + B Q B^{\top} $.
        
        Then,  the reduced order system \eqref{sys:r1} with reduced matrices \begin{align}
        \label{sys:r} 
    &\hat{A}= {P}_{2, \cD}^{\top}{P}_{1,\cD}^{-1}A{P}_{2, \cD}{P}_{3, \cD}^{-1}, \nonumber \\ 
    &\hat{B}={P}_{2, \cD}^{\top}{P}_{1,\cD}^{-1}B,~\text{and}~\hat{C}=C{P}_{2, \cD}{P}_{3, \cD}^{-1}
        \end{align}
        is asymptotically stable and satisfies 
    $\lim_{k \to \infty}\bE[(y_k - \hat{y}_k)^{\top}(y_k - \hat{y}_k)] \leq \gamma^{*}$, where $\gamma^{*}$ is the optimal value given by solving \eqref{eq:opt1}.
\end{secthm}
\begin{proof} The proof is in Appendix~\ref{app:1}.
\end{proof}

The upper bound $\gamma^{*}$ for asymptotical approximation error can only be determined after solving the optimization \eqref{eq:opt1}. If the obtained $\gamma^{*}$ is unsatisfactory, this suggests that the system order has been overly reduced. In such a case, a larger 
$r$ can be selected. Furthermore,  the constraint \eqref{con_4} is non-convex, making the optimization \eqref{eq:opt1} computationally inefficient. Nevertheless, this optimization problem remains a fundamental step in addressing distributionally robust model order reduction in Subsection \ref{subsec:drmor}. The issue of computational efficiency will be addressed in Section \ref{sec:algo}.

In fact, the optimization problem \eqref{eq:opt1} does not have a compact solution set and should, therefore, be expressed in terms of an infimum. However, the solution set can be reformulated as a compact set by replacing the strict inequality with a non-strict inequality using a sufficiently small constant. This reformulation enables the use of a standard minimization formulation instead of the infimum. Consequently, we adopt the formulation presented in \eqref{eq:opt1}.

\subsection{Distributionally Robust Model Order Reduction for Uncertain Distribution}
\label{subsec:drmor}
As observed in Theorem \ref{thm:opt}, the distribution $\bP$ plays a crucial role in model order reduction due to its function to constrain the solution set.  Therefore, an uncertain $\bP$ in the Wasserstein ambiguity set $\mathbb{W}_{\rho}(\bar{\bP})$ in \eqref{eq:P_dis} can render the guarantee for model order reduction performance invalid, which is undesirable.

Intuitively, this problem can be modeled by the following Stackelberg game.
\begin{align}
 \label{eq:game}
 \gamma^{\star} = \begin{cases}
        &\min_{P_{\cD}, Z_{\cD}, \gamma}\max_{\bP \in \mathbb{W}_{\rho}(\bar{\bP})} \  \gamma \\
        &\text{s.t.} \ (P_{\cD}, Z_{\cD}) \in \mathbb{B}(Q),
 \end{cases}
 \end{align}
where $\mathbb{B}(Q) := \{(P_{\cD}, Z_{\cD}, \gamma) \mid \eqref{con_1}-\eqref{con_5} \text{ hold for a given covariance } Q \text{ of } \bP\}$ is the solution set.
This problem can be interpreted as follows: the system designer first constructs the reduced-order model, after which the input seeks the worst-case perturbation that maximizes $\gamma$. Directly solving this Stackelberg game is difficult in general since the Wasserstein distance does not have an explicit closed expression~\cite{panaretos2019statistical}.

To provide a solution, we first define the Gelbrich distance on the space of convariance matrices as follows,
\begin{secdefn}
\label{def:gdis}
The Gelbrich distance between the two covariance matrices $\Gamma_1$ and $\Gamma_2 \in \mathbb{S}_{+}^m$ is given by
    \begin{align*}
    G(\Gamma_1, \Gamma_2) = \sqrt{\tr(\Gamma_1 + \Gamma_2 - 2(\Gamma_2^{\frac{1}{2}}\Gamma_1\Gamma_2^{\frac{1}{2}})^{\frac{1}{2}})}.
\end{align*}
\end{secdefn}
Then, the following Gelbrich ambiguity set for the covariance $\bar{Q}$ is defined by,
\begin{align*}
    \mathbb{G}_{\rho}(\bar{Q}) := \{Q \in \mathbb{S}_{+}^{m} \mid \bE_{x \sim \bP}[x] = 0, G(Q, \bar{Q}) \leq \rho \}.
\end{align*}
We are interested in Gelbrich distance since its connection to the 2-Wasserstein distance, as clarified by the following Lemma.
\begin{seclem}
\label{lem:Gdis}
    For any two distributions $\bP_1$ and $\bP_2$ on $\bR^n$ with the same mean and the covariance matrices $\Gamma_1, \Gamma_2 \in \mathbb{S}_{+}^m$, respectively, it holds that $\mathcal{W}_2(\bP_1, \bP_2) \geq G(\Gamma_1, \Gamma_2)$.
\end{seclem}
\begin{proof}
    The proof can be found in \cite{gelbrich1990formula}.
\end{proof}
As a direct consequence of Lemma~\ref{lem:Gdis}, we can conclude that $\mathbb{G}_{\rho}(\bar{Q})$ constitutes an outer approximation for the Wasserstein ambiguity set $\mathbb{W}_{\rho}(\bar{P})$, as summarized below.
\begin{seccor}
    $\mathbb{W}_{\rho}(\bar{\bP}) \subseteq \mathbb{G}_{\rho}(\bar{Q})$
\end{seccor}
\begin{proof}
    The proof follows directly from Lemma \ref{lem:Gdis}.
\end{proof}

Because $\mathbb{G}_{\rho}(\bar{Q})$ covers $\mathbb{W}_{\rho}(\bar{\bP})$, we can henceforth construct an upper bound for $\gamma^{\star}$ as follows,
\begin{align}
\label{eq:upgame}
 \bar{\gamma}^{\star} = \begin{cases}
        &\min_{P_{\cD}, Z_{\cD},\gamma}\max_{\bP \in\mathbb{G}_{\rho}(\bar{Q})} \  \gamma \\
        &\text{s.t.} \ (P_{\cD}, Z_{\cD}, \gamma) \in \mathbb{B}(Q).
 \end{cases}
 \end{align}
 Surprisingly, the following theorem establishes that the relaxation \eqref{eq:upgame} results in the same optimal value.
\begin{secthm}
\label{thm:opt_equiv}
    For the Stackelberg games \eqref{eq:game} and \eqref{eq:upgame}, the optimization values satisfy ${\gamma}^{\star} = \bar{\gamma}^{\star}$.
\end{secthm}

\begin{proof}
    The proof is in Appendix~\ref{app:2}
\end{proof}

Thanks to Theorem~\ref{thm:opt_equiv}, it can be concluded that the Stackelberg game~\eqref{eq:game} is equivalent to the Stackelberg game~\eqref{eq:upgame}, where the inner maximization is only related with the covariance $Q$. Therefore, we focus on how to solve the Stackelberg game~\eqref{eq:upgame} in the reminder of this section. By similar arguments in \cite{stein2003bi}, it is possible to transform the Stakelberg game into an equivalent semi-infinite programming  as below,
\begin{align}
\label{eq:semi_opt}
 {\gamma}^{\star} = \begin{cases}
        &\min_{P_{\cD},  Z_{\cD}, \gamma}\  \gamma \\
        &\text{s.t.} \ (P_{\cD}, Z_{\cD}, \gamma) \in \mathbb{B}(Q), \forall Q \in \mathbb{G}_{\rho}(\bar{Q}).
 \end{cases}
 \end{align}
Unfortunately, this semi-infinite programming problem remains unsolvable due to the presence of infinitely many constraints. To address this challenge, we seek to reformulate the problem into a finite-dimensional optimization problem. To achieve this, we first introduce a tractable alternative set of constraints for $Q_{\Delta}:=Q-\bar{Q}$. Subsequently, we employ a nested optimization approach to solve the original semi-infinite optimization problem~\eqref{eq:semi_opt}. The following Lemma characterizes the solution set for $Q_{\Delta}$.
 \begin{seclem} 
 \label{lem:constraint}
 The solution set for $Q_{\Delta}$ is a set constrained by a Semi-definite Programming (SDP) condition, i.e.,
      \begin{align*}
          Q_{\Delta} \in \{&Q_{\Delta} \in \mathbb{S}^{m} \mid \exists E_{Q} \in \mathbb{S}_{+}^{m} \text{ with }  \\
          &\tr(Q_{\Delta}+2\bar{Q} - 2E_{Q}) \leq \rho^2, \\
          &\begin{bmatrix}
              \bar{Q}^{\frac{1}{2}} Q_{\Delta} \bar{Q}^{\frac{1}{2}} + \bar{Q}^{2} & E_{Q}\\
              E_{Q} & I
          \end{bmatrix} \succeq 0\}.
      \end{align*}
 \end{seclem}
 \begin{proof}
     The proof is in Appendix~\ref{app:3}.
 \end{proof}

 Since $Q_{\Delta}$ is constrained by the SDP condition, it is possible to find the worst $Q_{\Delta}$ using the semi-infinite programming for~\eqref{eq:semi_opt}. Afterwards, one can construct a nested optimization to solve \eqref{eq:semi_opt} and hence solve~\eqref{eq:game}. The next theorem summarizes the nested optimization approach to solve~\eqref{eq:game}.
\begin{secthm}
    The optimal value $\tilde{\gamma}^{\star}$ of Algorithm~\ref{alg:opt1} provides an upper bound for $\gamma^{\star}$ in~\eqref{eq:game}, i.e.,
 \begin{align*}
     \tilde{\gamma}^{\star} \geq \gamma^{\star}.
 \end{align*}
\end{secthm}
\begin{proof}
    It is easy to check that $\bar{Q}+ \beta^{\star}I \succeq Q$ and $\mathbb{B}(\bar{Q}+ \beta^{\star}I) \subseteq  \mathbb{B}(Q), \forall Q \in \mathbb{G}_{\rho}(\bar{Q})$. Therefore, $\tilde{\gamma}^{\star} \geq {\gamma}^{\star}$ holds.
\end{proof}

Since we have $\mathbb{B}(\bar{Q}+ \beta^{\star}I) \subseteq  \mathbb{B}(Q), \forall Q \in \mathbb{G}_{\rho}(\bar{Q})$, the optimal solution obtained by Algorithm \ref{alg:opt} is a suboptimal solution for the Stackelberg game~\eqref{eq:game}. Therefore, we can use Algorithm \ref{alg:opt1} to obtain a reduced order model~\eqref{sys:r1} with matrices in \eqref{sys:r}. Furthermore, Algorithm \ref{alg:opt1} provides an explicit \textit{a priori} upper bound for the approximation error even if $\bP$ is uncertain.

\begin{algorithm}
\caption{Nested Optimization for Distributionally Robust Model Order Reduction}\label{alg:opt1}
\begin{algorithmic}
\Require Define the reduced-order $r$, the system triplet $(A,B,C)$. \\
Step 1. Solve the following convex optimization problem and get optimal value $\beta^{\star}$. \begin{align*}
    &\max_{E_{Q}} \  \tr (Q_{\Delta}) \\
        &\text{s.t.} \ \tr(Q_{\Delta}+2\bar{Q} - 2E_{Q}) \leq \rho^2,\\
        & \ \ \ \begin{bmatrix}
              \bar{Q}^{\frac{1}{2}} Q_{\Delta} \bar{Q}^{\frac{1}{2}} + \bar{Q}^{2} & E_{Q}\\
              E_{Q} & I
          \end{bmatrix} \succeq 0.
\end{align*}\\
Step 2. Solve the following optimization problem with respect to $P_{\cD} \in \mathbb{S}_{++}^{n}$, $Z_{\cD}\in \mathbb{S}_{++}^{r}$ and $\gamma$.
    \begin{align*}
    \tilde{\gamma}^{\star} =\begin{cases}
        &\min_{P_{\cD}, Z_{\cD}, \gamma}  \  \gamma \nonumber \\
        &\text{s.t.} \   P_{\cD} = \begin{bmatrix}
            P_{1,\cD} & P_{2,\cD} \\
            P_{2,\cD}^{\top} & P_{3,\cD}
        \end{bmatrix} \succ 0, \\
        &P_{1,\cD} \in \mathbb{S}^{n}_{+}, P_{3, \cD} \in \mathbb{S}^{r}_{+},\\
        &\Psi(P_{\cD}, Z_{\cD}, Q + \beta^{\star}I) = \begin{bmatrix}
            \Psi_1 & \Psi_2 \\
            \Psi_2^{\top} & \Psi_3
        \end{bmatrix} \prec 0, \\
        &  Z_{\cD} = P_{2,\cD} P_{3, \cD}^{-1} P_{2,\cD}^{\top}, \operatorname{rank}(Z_{\cD}) = r,\\
     &\tr\left(C ( P_{1, \cD} -  Z_{\cD}) C^{\top}\right) \leq \gamma,
     \end{cases}
    \end{align*}
\\
Step 3. Write down the reduced order system dynamics as according to \eqref{sys:r}.
\end{algorithmic}
\end{algorithm}
 
  \section{Tractable Convex Optimization Solution}
  \label{sec:algo}
  In the previous section, we have relaxed the distributionally robust model order reduction problem into a nested optimization problem. However, $(P_{\cD}, Z_{\cD}) \in \mathbb{B}(\bar{Q}+ \beta^{\star}I)$ is not a convex constraint, which prohibits efficient solutions. In order to have a tractable solution, Algorithm~\ref{alg:opt1} is relaxed  into a convex nested optimization, as summarized by Algorithm \ref{alg:opt}. We refer to this approach as distributionally robust optimization-based model order reduction (DROMOR).
\begin{algorithm}
\caption{Distributionally Robust Optimization-Based Model Order Reduction }\label{alg:opt}
\begin{algorithmic}
\Require Define the reduced-order $r$, the system triplet $(A,B,C)$. \\
Step 1. Solve the following convex optimization problem and get optimal value $\beta^{\star}$. \begin{align*}
    &\max_{E_{Q}, Q_{\Delta}} \  \tr (Q_{\Delta}) \\
        &\text{s.t.} \ \tr(Q_{\Delta}+2\bar{Q} - 2E_{Q}) \leq \rho^2,\\
        & \ \ \ \begin{bmatrix}
              \bar{Q}^{\frac{1}{2}} Q_{\Delta} \bar{Q}^{\frac{1}{2}} + \bar{Q}^{2} & E_{Q}\\
              E_{Q} & I
          \end{bmatrix} \succeq 0.
\end{align*}\\
Step 2. Solve the following convex optimization problem with respect to $P_{\cD} \in \mathbb{S}_{++}^{n}$, $Z_{\cD}\in \mathbb{S}_{++}^{r}$ and $\gamma$.
    \begin{align*}
    \tilde{\gamma}^{\star} =\begin{cases}
        &\min_{P_{\cD}, Z_{\cD}}  \  \gamma \nonumber \\
        &\text{s.t.} \   Z_{\cD} = \begin{bmatrix}
            Z_{1,\cD} & 0 \\
            0 & 0
        \end{bmatrix} \succeq 0, \\
        &P_{1,\cD} \succ 0, Z_{1,\cD} \succ 0.\\
        &\Psi(P_{\cD}, Z_{\cD}, Q + \beta^{\star}I) = \begin{bmatrix}
            \Psi_1 & \Psi_2 \\
            \Psi_2^{\top} & \Psi_3
        \end{bmatrix} \prec 0, \\
        &  P_{1,\cD} - Z_{\cD} \succ 0, \\
     &\tr\left(C ( P_{1, \cD} -  Z_{\cD}) C^{\top}\right) \leq \gamma,
     \end{cases}
    \end{align*}
\\
Step 3. Obtain the Schur decomposition of the matrix $Z_{1,\cD}$, i.e., $Z_{1,\cD} \gets U T U^{\top}$.\\
Step 4. $P_{2,\cD} \gets \begin{bmatrix}
    U \\ 0
\end{bmatrix}$ and $P_{3,\cD} \gets T^{-1}$.\\
Step 5. Write down the reduced order system dynamics as according to \eqref{sys:r}.
\end{algorithmic}
\end{algorithm}

In fact, we convexify the problem by choosing the structure of the matrix $Z_{\cD}$. Generally, there are many other possibilities to choose the structure of the matrix $Z_{\cD}$. The structure we choose is suitable when the pair $(A,C)$ is in the observable canonical form \cite{ibrir2018projection}. Therefore, the original system~\eqref{sys:origin} can be transformed to the observable canonical form first and then  Algorithm~\ref{alg:opt} can be used to obtain a reduced-order model. For more discussions about the choice of $Z_{\cD}$, see \cite{ibrir2018projection}.
\section{Numerical Experiment}
\label{sec:simu}
In this section, we use a numerical example to validate the results proposed in this paper. We employ a standard mechanical system as described in \cite{cheng2019balanced}, where we define $C = \begin{bmatrix}
    1 & 0 & 0 & 0\end{bmatrix}$, $D = \begin{bmatrix}
    5 & -1 \\ -1 & 5
\end{bmatrix}$ and $M = \frac{1}{2}I_2$
. For further details, refer to \cite{cheng2019balanced}. Using zero order hold discretization, the corresponding discrete-time system is given by
\begin{align*}
    &A = \begin{bmatrix}
        0.82 & -0.02 & 0.17 & 0.03 \\
     -0.02 & 0.82 & 0.03 & 0.17\\
     -0.08 & -0.01 & -0.01 & -0.01\\
     -0.01 & -0.08 & -0.01 & -0.02
    \end{bmatrix},  
    \end{align*}
    \begin{align*}
B = \begin{bmatrix}
        0.17 & 0.03 \\ 0.03 & 0.17 \\ 0.09 & 0.02 \\ 0.02 & 0.09
    \end{bmatrix}, C = \begin{bmatrix}
        1 & 0 & 0 & 0
    \end{bmatrix}.
\end{align*}
The center of Wasserstein ambiguity set in \eqref{eq:P_dis} is defined as $\bar{\bP} = \mathcal{N}(0, \bar{Q})$ with $\bar{Q} = \begin{bmatrix}
    0.01 & 0 \\ 0 & 1
\end{bmatrix}$. Suppose $\rho^2 = 2$ and the actual distribution $\bP$ is defined to be $\bP = \mathcal{N}(0, {Q})$, where ${Q} = \begin{bmatrix}
    1 & 0 \\ 0 & 0.01
\end{bmatrix}$. It can be checked that $\bP \in \mathbb{W}_{\rho}(\bar{\bP})$. Then, the reduced order model can be calculated according to Algorithm \ref{alg:opt} as follows,
\begin{align*}
    &\hat{A} = \begin{bmatrix}
        0.78 &	-0.03 \\
-0.03 &	0.83
    \end{bmatrix}, \hat{B} = 10^{-9}\times\begin{bmatrix}
        -4.04 & -8.83\\ 4.00 & 0.83
    \end{bmatrix}, \\
    &\hat{C} = 10^{7}\times\begin{bmatrix}
        -0.62 & -2.2
    \end{bmatrix}.
\end{align*}

To show the effectiveness of the result, we also generate reduced order model \eqref{sys:r1} by balanced truncation, where the reduced order $r=2$ and the balanced truncation method use $\bar{\bP}$ directly. The approximation results are shown in Figure \ref{fig:y}. The absolute approximation error with the two methods are compared in Figure \ref{fig:e}.

\begin{figure}
    \centering
    \includegraphics[width=1\linewidth,height=4cm]{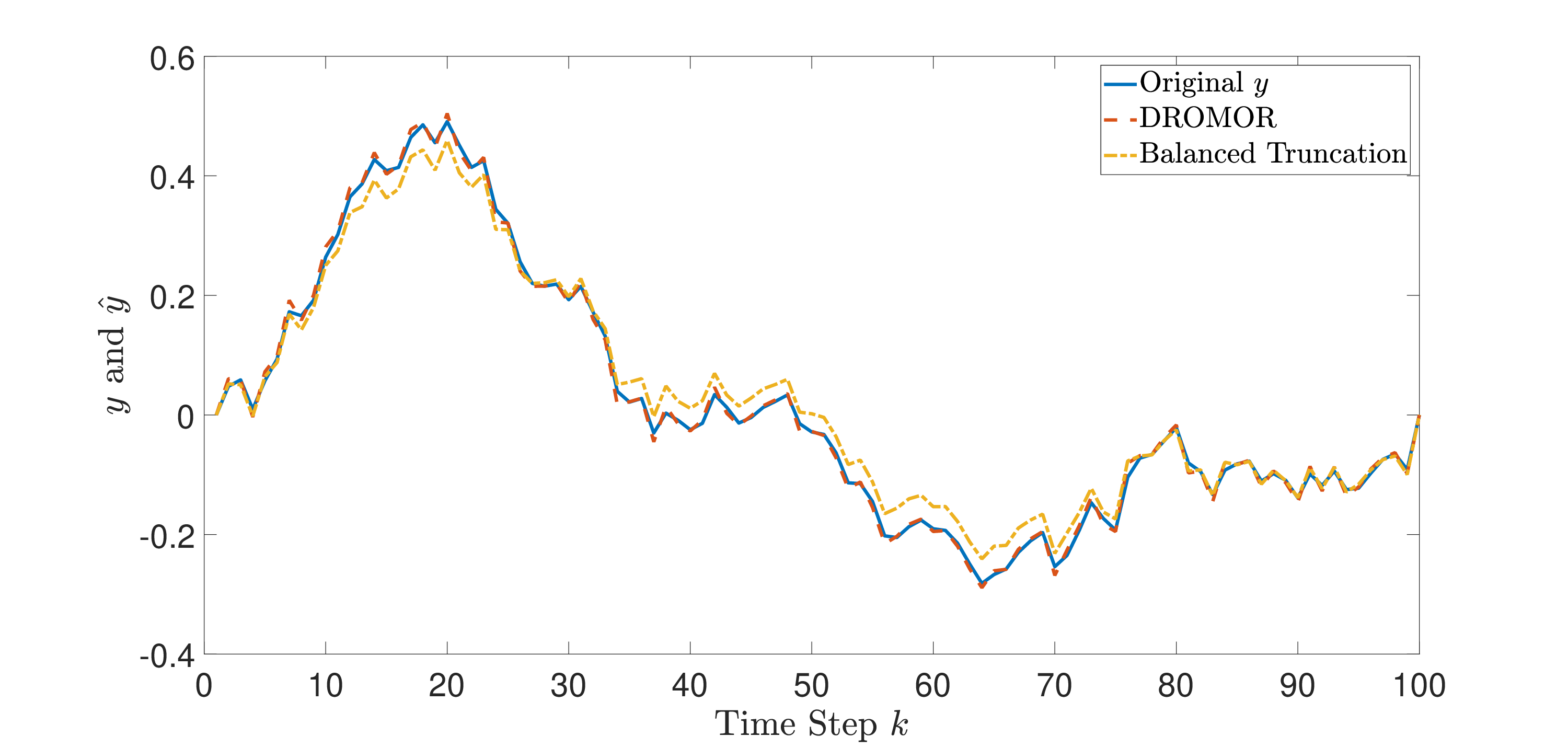}
    \caption{Performance Results by DROMOR and Balanced Truncation.}
    \label{fig:y}
\end{figure}

\begin{figure}
    \centering
    \includegraphics[width=1\linewidth,height=4cm]{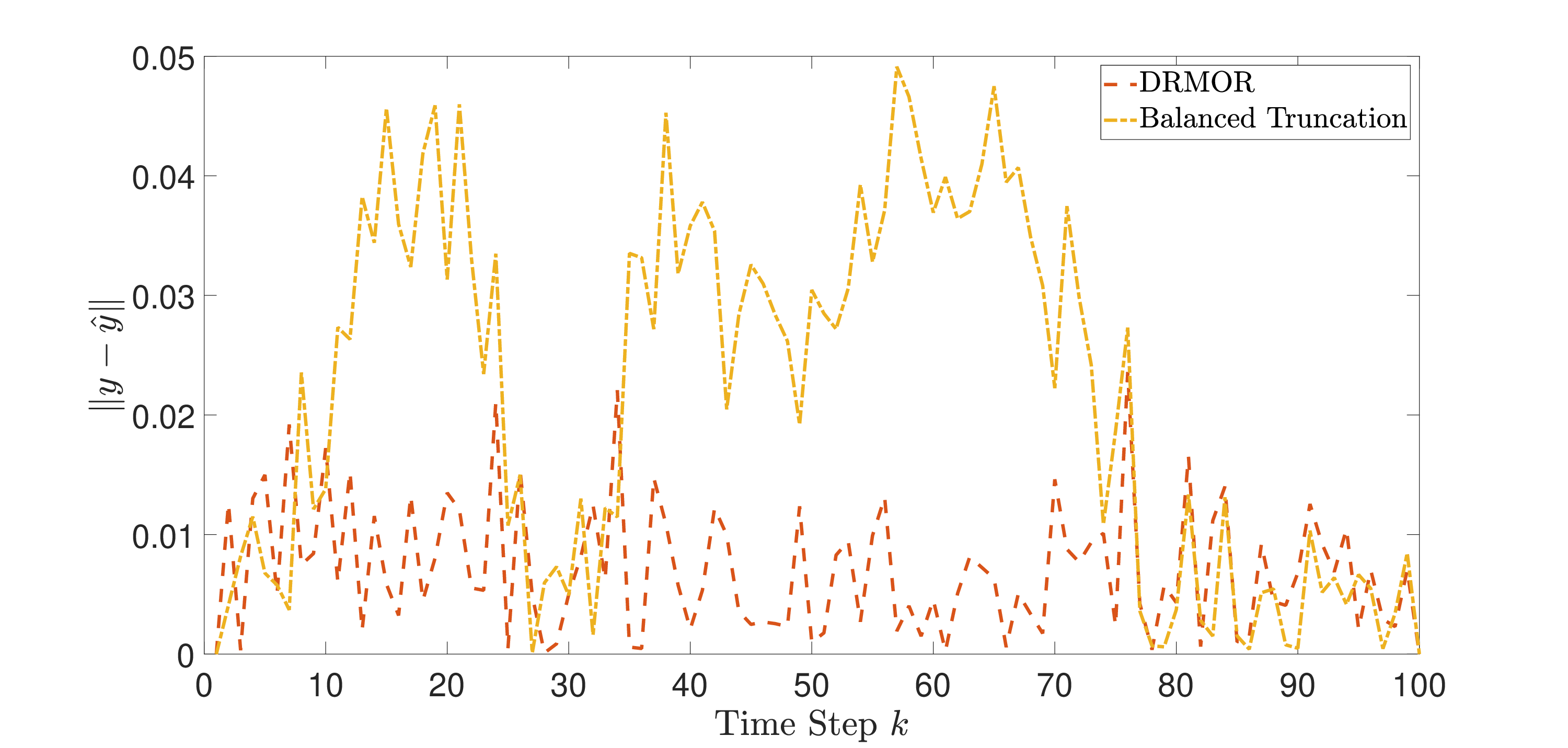}
    \caption{Absolute Approximation Error by DROMOR and Balanced Truncation.}
    \label{fig:e}
\end{figure}

From Figure \ref{fig:y} and \ref{fig:e}, it can be observed the approximation performance with DROMOR is better. We can conclude that DROMOR shows a better performance than balanced truncation when the distribution of the inputs is uncertain.

\section{Conclusion}
\label{sec:con}
In this paper, we have established a novel optimization method for distributionally robust model order reduction. Although the original problem is a intractable Stackelberg game, we have relaxed it into a nested convex optimization. Moreover, the error bound for the system has been provided by the result of the nested convex optimization. For future works, distributionally robust model order reduction for nonlinear system is of interest.
\appendices
\section{Proof of Theorem \ref{thm:opt}}
\label{app:1}
 First, we show that $\hat{A}$ is asymptotically stable. Consider the dynamics of the following augmented system, which consists of the states of the original system \eqref{sys:origin} and the reduced-order model \eqref{sys:r1}, and can be expressed as
\begin{subequations}
    \begin{align*}
  \begin{bmatrix}
      x_{k+1} \\ \hat{x}_{k+1} 
  \end{bmatrix} = A_{\Delta} \begin{bmatrix}
      x_{k} \\ \hat{x}_{k} 
  \end{bmatrix} + B_{\Delta}u_k,
\end{align*}
\begin{align*}
    e_{k}: = y_{k} - \hat{y}_k = C_{\Delta}\begin{bmatrix}
      x_{k} \\ \hat{x}_{k}
  \end{bmatrix},
\end{align*}
\end{subequations}
where $A_{\Delta} = \begin{bmatrix}
    A & 0 \\ 0 & \hat{A}
\end{bmatrix}$, $B_{\Delta} = \begin{bmatrix}
     B \\ \hat{B}
\end{bmatrix}$ and $C_{\Delta} = \begin{bmatrix}
    C & -\hat{C}
\end{bmatrix}$. 
From the Lyapunov theory, we can show that $A_{\Delta}$ is asymptotically stable if there exists a matrix $P_{\Delta} \in \mathbb{S}_{++}^{n}$ such that
\begin{align}
\label{eq:glya}
    A_{\Delta} \hat{P}_{\Delta} A_{\Delta}^{\top} - \hat{P}_{\Delta} + B_{\Delta} Q B_{\Delta}^{\top} \prec 0.
\end{align}
Next, we set $\hat{P}_{\Delta} = P_{\cD} $ and verify that $\hat{P}_{\Delta}$ satisfies \eqref{eq:glya}. We can show that \eqref{eq:glya} is equivalent to
\begin{align*}
    \Theta := \begin{bmatrix}
        \Theta_1 & \Theta_2\\
        \Theta_2^{\top} & \Theta_3
    \end{bmatrix} \prec 0 ,
\end{align*}
where
\begin{align*}
    \Theta_1 =& A P_{1,\cD} A^{\top} -  P_{1,\cD}+ B Q B^{\top}, \\
    \Theta_2 =& A P_{2,\cD} \hat{A}^{\top} -  P_{2,\cD}+ B Q \hat{B}^{\top},  \\
    \Theta_3 =& \hat{A} P_{3,\cD} \hat{A}^{\top} -  P_{3,\cD}+ \hat{B} Q \hat{B}^{\top}.
\end{align*}
By Schur Complement Lemma and the positive definiteness of $P_{\cD}$, it yields that $P_{3,\cD} \succ P_{2,\cD}P_{1,\cD}^{-1}P_{2,\cD}^{\top}$. Therefore, we have 
\begin{align}
\label{eq:1_1}
    \begin{bmatrix}
        \Theta_1 & \Theta_2\\
        \Theta_2^{\top} & \Theta_3
    \end{bmatrix} \preceq \begin{bmatrix}
        \Theta_1 & \Theta_2\\
        \Theta_2^{\top} & \Theta_3 + P_{3,\cD} - P_{2,\cD}P_{1,\cD}^{-1}P_{2,\cD}^{\top}
    \end{bmatrix}.
\end{align}
Recall  that $\hat{A} = P_{2, \cD}^{\top} P_{1, \cD}^{-1} A P_{2, \cD} P_{3, \cD}^{-1}$ and ${P}_{2, \cD}^{\top}{P}_{1,\cD}^{-1}B$, the right hand side of~\eqref{eq:1_1} is equivalent to
\begin{align*}
    &\begin{bmatrix}
        I_n & 0 \\
        0 &  P_{2, \cD}^{\top}P_{1, \cD}^{-1}
    \end{bmatrix}\begin{bmatrix}
        \Psi_1 & \Psi_2\\
        \Psi_2^{\top} & \Psi_3
    \end{bmatrix} \begin{bmatrix}
         I_n & 0 \\
        0 & P_{1, \cD}^{-1} P_{2, \cD} 
    \end{bmatrix} \prec 0.
\end{align*}
Consequently, the matrix $\Theta$ is negative definite if $\Psi( P_{\cD}, Z_{\cD}, Q)$ is negative definite. Since condition \eqref{con_2} holds, we have established that $A_{\Delta}$ is asymptotically stable and hence $\hat{A}$ is asymptotically stable.

Next, we prove that $\hat{P}_{\Delta} - P_{\Delta}$ is positive definite, where $P_{\Delta}$ is the solution of the Lyapunov equation for $(A_{\Delta}, B_{\Delta})$, i.e., 
\begin{align}
    \label{eq:lyap} A_{\Delta} P_{\Delta} A_{\Delta}^{\top} - P_{\Delta} + B_{\Delta}B_{\Delta}^{\top} = 0,
\end{align} where the existence is guaranteed by the asymptotical stability of $A_{\Delta}$. This result will be instrumental in establishing an upper bound for $\lim_{k \to \infty}\bE [(y_k - \hat{y}_k)^{\top}(y_k - \hat{y}_k)]$.

By making the difference between \eqref{eq:glya} and \eqref{eq:lyap}, we have
\begin{align}
\label{eq:dif_lya}
    A_{\Delta} (\hat{P}_{\Delta} - P_{\Delta}) A_{\Delta}^{\top} + \hat{P}_{\Delta} - P_{\Delta} \prec 0,
\end{align}
 Since $A_{\Delta}$ is asymptotically stable, we can conclude from \eqref{eq:dif_lya} and the Lyapunov theory that $\hat{P}_{\Delta} - P_{\Delta}$ is positive definite. 

 In the last step, we prove that $\lim_{k \to \infty}\bE [(y_k- \hat{y}_k)^{\top}(y_k - \hat{y}_k)] \leq \gamma$. A direct computation yields
\begin{align*}
    &\lim_{k \to \infty}\bE [(y_k - \hat{y}_k)^{\top}(y_k - \hat{y}_k)] = \tr(C_{\Delta} P_{\Delta} C_{\Delta}^{\top}) \\
    = & \tr (C_{\Delta}\hat{P}_{\Delta}C_{\Delta}^{\top}) - \tr (C_{\Delta}(\hat{P}_{\Delta}- P_{\Delta})C_{\Delta}^{\top}) \\
    \leq & \tr (C_{\Delta}\hat{P}_{\Delta}C_{\Delta}^{\top})\\
    = & \tr (CP_{1,\cD}C^{\top})+ \tr (\hat{C}P_{3,\cD}\hat{C}^{\top}),
\end{align*}
where the inequality holds since $\hat{P}_{\Delta} - P_{\Delta} \succeq 0$. Substituting $\hat{C} = C P_{2,\cD}P_{3,\cD}^{-1}$ into the last equation yields
\begin{align*}
    &\tr ( CP_{1,\cD}C^{\top})+ \tr ( \hat{C}P_{3,\cD}\hat{C}^{\top}) \\
    =& \tr ( CP_{1,\cD}C^{\top}) - \tr ( CP_{2, \cD}P_{3,\cD}^{-1}P_{2, \cD}^{\top}C^{\top} )\\
    = & \tr \left(C( P_{1, \cD} -  Z_{\cD}) C^{\top}\right) \leq \gamma,
\end{align*}
which completes the proof.
\red
\section{Proof of Theorem \ref{thm:opt_equiv}}
\label{app:2}
We first establish a lower bound for $\gamma^{\star}$ by constructing the following optimization problem.
    \begin{align}
        \label{eq:downgame}
        \underline{\gamma}^{\star} = \begin{cases}
        &\min_{P_{\cD},  Z_{\cD}}\max_{\bP \in \mathbb{W}_{\rho}(\bar{\bP})\cap \mathcal{N}_0} \  \gamma \\
        &\text{s.t.} \ (P_{\cD}, Z_{\cD}) \in \mathbb{B}(Q).
         \end{cases}
    \end{align}
    where $\mathcal{N}_0 : = \{\bP \mid \bP = \mathcal{N}(0,\Gamma), \Gamma \in \mathbb{S}_{+}^n\}$. Since $\mathbb{W}_{\rho}(\bar{\bP})\cap \mathcal{N}_0 \subseteq \mathbb{W}_{\rho}(\bar{\bP})$, the inner maximization is constrained to a smaller set, which yields that
    \begin{align*}
        \max_{\bP \in \mathbb{W}_{\rho}(\bar{\bP})\cap \mathcal{N}_0} \  \gamma \leq \max_{\bP \in \mathbb{W}_{\rho}(\bar{\bP})}  \gamma.
    \end{align*}
    This further implies that
    \begin{align*}
        \underline{\gamma}^{\star} \leq \gamma^{\star}.
    \end{align*}
    Recall that the 2-Wasserstein distance between two Gaussian distribution $\bP_1 = \mathcal{N}(0, \Gamma_1)$ and $\bP_2 = \mathcal{N}(0, \Gamma_2)$ satisfies 
\begin{align*}
    \mathcal{W}_2(\bP_1, \bP_2) = G(\Gamma_1, \Gamma_2),
\end{align*}
\eqref{eq:downgame} can be rewritten as
\begin{align*}
        \underline{\gamma}^{\star} = \begin{cases}
        &\min_{P_{\cD}, Z_{\cD},\gamma}\max_{Q} \  \gamma \\
        &\text{s.t.} \ G(Q, \bar{Q}) \leq \rho\\
        &\ \ \ (P_{\cD}, Z_{\cD}) \in \mathbb{B}(Q),
         \end{cases}
    \end{align*}
    which is equivalent to the Stackelberg game~\eqref{eq:upgame}. Therefore, we have
    \begin{align*}
        \underline{\gamma}^{\star} = \bar{\gamma}^{\star}.
    \end{align*}This, together with $\gamma^{\star} \leq \bar{\gamma}^{\star}$, yields
    $\underline{\gamma}^{\star} = \gamma^{\star} = \bar{\gamma}^{\star}$, which completes the proof.
    \red
    \section{Proof of Theorem \ref{lem:constraint}}
    \label{app:3}
    By the definition ~\ref{def:gdis} of Gelbrich distance, we have 
    \begin{align*}
        Q_{\Delta} \in \{Q_{\Delta} \in \mathbb{S}^m \mid \tr(&Q_{\Delta} + 2\bar{Q}  \\
        -& 2(\bar{Q}^{\frac{1}{2}}(\bar{Q} + Q_{\Delta})\bar{Q}^{\frac{1}{2}})^{\frac{1}{2}}) \leq \rho^2\},
    \end{align*}
    By introducing an auxiliary variable $E_{Q} \in \mathbb{S}^{m}_{+}$ such that 
    \begin{align}\label{ineq:EQ}
      E_{Q}^2 \preceq \bar{Q}^{\frac{1}{2}}Q\bar{Q}^{\frac{1}{2}},  
    \end{align}
    we can be further recast it as $E_Q \preceq (\bar{Q}^{\frac{1}{2}}Q\bar{Q}^{\frac{1}{2}})^{\frac{1}{2}}$. Therefore, the constraint for $Q_{\Delta}$ can be reformulated as $\tr(Q_{\Delta}+2\bar{Q} -2E_Q) \leq \rho^2$. Furthermore, using Schur Complment Lemma, \eqref{ineq:EQ} is equivalent to \begin{align*}
            \begin{bmatrix}
              \bar{Q}^{\frac{1}{2}} Q_{\Delta} \bar{Q}^{\frac{1}{2}} + \bar{Q}^{2} & E_{Q}\\
              E_{Q} & I
          \end{bmatrix} \succeq 0,
    \end{align*}
    which completes the proof.\red
\bibliographystyle{ieeetr}
\bibliography{ref}
\end{document}